\documentclass[12pt]{article}

\usepackage{amssymb}
\usepackage{amsmath}
\usepackage{graphicx}
\usepackage{pst-all}

\newtheorem{thm}{Theorem}[section]
\newtheorem{cor}[thm]{Corollary}
\newtheorem{lem}[thm]{Lemma}

\newenvironment{proof}[1][Proof]{\noindent\textbf{#1.} }{\hfill\rule{1mm}{2mm}}

\begin{document}
\title
{Cycles and Paths Embedded in Varietal Hypercubes
\thanks{The work was
supported by NNSF of China (No. 61272008).} }

\author
{ Jin Cao \quad Li Xiao \quad Jun-Ming Xu\footnote{Corresponding
author, E-mail address: xujm@ustc.edu.cn (J.-M. Xu)}
\\ \\
{\small Department of Mathematics}\\
{\small University of Science and Technology of China}\\
{\small  Wentsun Wu Key Laboratory of CAS}  \\
{\small Hefei, Anhui, 230026, China} }

\date{} \maketitle

\begin{minipage}{140mm}

\begin{center} {\bf Abstract} \end{center}

The varietal hypercube $VQ_n$ is a variant of the hypercube $Q_n$
and has better properties than $Q_n$ with the same number of edges
and vertices. This paper shows that every edge of $VQ_n$ is
contained in cycles of every length from 4 to $2^n$ except 5, and
every pair of vertices with distance $d$ is connected by paths of
every length from $d$ to $2^n-1$ except $2$ and $4$ if $d=1$.



\end{minipage}

\vskip0.6cm

\section{Introduction}

The hypercube network $Q_n$ has proved to be one of the most popular
interconnection networks since it has a simple structure and has
many nice properties. As a variant of $Q_n$, the varietal hypercube
$VQ_n$, proposed by Cheng and Chuang~\cite{cc94} in 1994, has many
properties similar or superior to $Q_n$. For example, the
connectivity and restricted connectivity of $VQ_n$ and $Q_n$ are the
same (see Wang and Xu~\cite{wx09}), while, all the diameter and the
average distance, fault-diameter and wide-diameter of $VQ_n$ are
smaller than that of the hypercube (see Cheng and
Chuang~\cite{cc94}, Jiang {\it et al.}~\cite{jhl10}).

Several topological structures of multicomputer systems are commonly
used in various applications such as image processing and scientific
computing. Among them, the most common structures are paths and
cycles. Embedding these structures in various well-known networks,
such as $Q_n$, have been extensively investigated in the literature
(see, for example, a survey by Xu and ma~\cite{xm09}). However,
embedding these structures in $VQ_n$ has been not investigated as
yet. In this paper, we show that $VQ_n$ should be capable of
embedding these structures. Main results can be stated as follows.

Every edge of $VQ_n$ is contained in cycles of every length from 4
to $2^n$ except 5, and every pair of vertices with distance $d$ is
connected by paths of every length from $d$ to $2^n-1$ except $2$
and $4$ if $d=1$.

The proofs of these results are in Section 3. The definition and
some basic properties of $VQ_n$ are given in Section 2.

\section{Definitions and Lemmas}

We follow~\cite{x03} for graph-theoretical terminology and notation
not defined here. A graph $G=(V,E)$ always means a simple and
connected graph, where $V=V(G)$ is the vertex-set and $E=E(G)$ is
the edge-set of $G$. For $uv\in E(G)$, we call $u$ (resp. $v$) is a
neighbor of $v$ (resp. $u$). A {\it $uv$-path} is a sequence of
adjacent vertices, written as $(v_0,v_1,v_2,\cdots,v_m)$, in which
$u=v_0$, $v=v_m$ and all the vertices $v_0,v_1,v_2,\cdots,v_m$ are
different from each other, $u$ and $v$ is called the {\it
end-vertices} of $P$. If $u=v$, then a $uv$-path $P$ is called a
{\it cycle}. The {\it length} of a path $P$, denoted by
$\varepsilon(P)$, is the number of edges in $P$. The length of a
shortest $uv$-path in $G$ is called the {\it distance} between $u$
and $v$ in $G$, denoted by $d_G(u,v)$. For a path
$P=(v_0,v_1,\cdots,v_i, v_{i+1}, \cdots, v_m)$, we can write
$P=P(v_0,v_i)+v_iv_{i+1}+P(v_{i+1},v_m)$, and the notation
$P-v_iv_{i+1}$ denotes the subgraph obtained from $P$ by deleting
the edge $v_iv_{i+1}$.

\vskip30pt

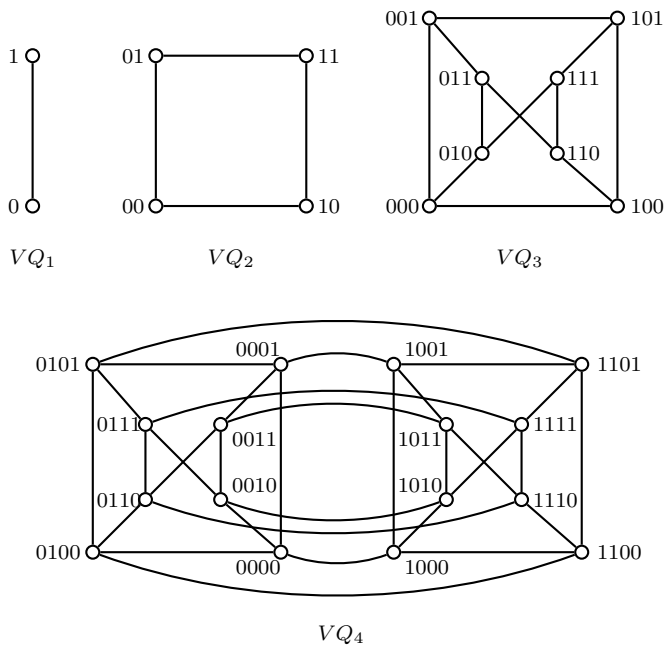
\begin{figure*}[ht]
\begin{pspicture}(-2.3,-.5)(1,3)
\cnode(1,1){.1}{0}\rput(.75,1){\scriptsize0}
\cnode(1,3){.1}{1}\rput(.75,3){\scriptsize1}
\ncline{0}{1}\rput(1,.3){\scriptsize$VQ_1$}
\end{pspicture}
\begin{pspicture}(-.5,-.5)(3,3)
\cnode(1,1){.1}{00}\rput(.7,1){\scriptsize00}
\cnode(1,3){.1}{01}\rput(.7,3){\scriptsize01}
\cnode(3,1){.1}{10}\rput(3.3,1){\scriptsize10}
\cnode(3,3){.1}{11}\rput(3.3,3){\scriptsize11}
\ncline{00}{01}\ncline{01}{11}\ncline{11}{10}\ncline{10}{00}
\rput(2,.3){\scriptsize$VQ_2$}
\end{pspicture}
\begin{pspicture}(-.5,-.5)(3,3)
\cnode(1,1){.1}{000}\rput(.64,1){\scriptsize000}
\cnode(1,3.5){.1}{001}\rput(.64,3.5){\scriptsize001}
\cnode(3.5,1){.1}{100}\rput(3.9,1){\scriptsize100}
\cnode(3.5,3.5){.1}{101}\rput(3.9,3.5){\scriptsize101}
\cnode(1.7,1.7){.1}{010}\rput(1.35,1.7){\scriptsize010}
\cnode(1.7,2.7){.1}{011}\rput(1.35,2.7){\scriptsize011}
\cnode(2.7,1.7){.1}{110}\rput(3.05,1.7){\scriptsize110}
\cnode(2.7,2.7){.1}{111}\rput(3.05,2.7){\scriptsize111}
\ncline{000}{001}\ncline{001}{101}\ncline{101}{100}\ncline{100}{000}
\ncline{010}{011}\ncline{011}{110}\ncline{111}{110}\ncline{111}{010}
\ncline{000}{010}\ncline{001}{011}\ncline{101}{111}\ncline{100}{110}
\rput(2.2,.3){\scriptsize$VQ_3$}
\end{pspicture}

\vskip2pt
\begin{pspicture}(-3.1,0)(8,4)

\cnode(1,1){.1}{0100}\rput(.54,1){\scriptsize0100}%
\cnode(1,3.5){.1}{0101}\rput(.54,3.5){\scriptsize0101}%
\cnode(3.5,1){.1}{0000}\rput(3.2,0.8){\scriptsize0000}%
\cnode(3.5,3.5){.1}{0001}\rput(3.2,3.7){\scriptsize0001}%
\cnode(1.7,1.7){.1}{0110}\rput(1.35,1.7){\scriptsize0110}%
\cnode(1.7,2.7){.1}{0111}\rput(1.35,2.7){\scriptsize0111}%
\cnode(2.7,1.7){.1}{0010}\rput(3.15,1.9){\scriptsize0010}%
\cnode(2.7,2.7){.1}{0011}\rput(3.15,2.5){\scriptsize0011}

\ncline{0000}{0001}\ncline{0001}{0101}\ncline{0101}{0100}\ncline{0100}{0000}
\ncline{0010}{0011}\ncline{0011}{0110}\ncline{0111}{0110}\ncline{0111}{0010}
\ncline{0000}{0010}\ncline{0001}{0011}\ncline{0101}{0111}\ncline{0100}{0110}

\cnode(5,1){.1}{1000}\rput(5.44,0.8){\scriptsize1000}
\cnode(5,3.5){.1}{1001}\rput(5.44,3.7){\scriptsize1001}
\cnode(7.5,1){.1}{1100}\rput(8.,1){\scriptsize1100}
\cnode(7.5,3.5){.1}{1101}\rput(8.0,3.5){\scriptsize1101}
\cnode(5.7,1.7){.1}{1010}\rput(5.35,1.9){\scriptsize1010}
\cnode(5.7,2.7){.1}{1011}\rput(5.35,2.5){\scriptsize1011}
\cnode(6.7,1.7){.1}{1110}\rput(7.15,1.7){\scriptsize1110}
\cnode(6.7,2.7){.1}{1111}\rput(7.15,2.7){\scriptsize1111}

\ncline{1000}{1001}\ncline{1001}{1101}\ncline{1101}{1100}\ncline{1100}{1000}
\ncline{1010}{1011}\ncline{1011}{1110}\ncline{1111}{1110}\ncline{1111}{1010}
\ncline{1000}{1010}\ncline{1001}{1011}\ncline{1101}{1111}\ncline{1100}{1110}

\nccurve[angleA=20,angleB=160]{0101}{1101}
\nccurve[angleA=20,angleB=160]{0111}{1111}
\nccurve[angleA=-20,angleB=-160]{0110}{1110}
\nccurve[angleA=-20,angleB=-160]{0100}{1100}
\nccurve[angleA=20,angleB=160]{0001}{1001}
\nccurve[angleA=20,angleB=160]{0011}{1011}
\nccurve[angleA=-20,angleB=-160]{0010}{1010}
\nccurve[angleA=-20,angleB=-160]{0000}{1000}
\rput(4.3,-0.1){\scriptsize$VQ_4$}

\end{pspicture}

\caption{\label{f1}{\footnotesize \ The varietal hypercubes $VQ_1,
VQ_2, VQ_3$ and $VQ_4$}}
\end{figure*}

The {\it $n$-dimensional varietal hypercube} $VQ_n$ is the labeled
graph defined recursively as follows. $VQ_1$ is the complete graph
of two vertices labeled with $0$ and $1$, respectively. Assume that
$VQ_{n-1}$ has been constructed. Let $VQ^0_{n-1}$ (resp.
$VQ^1_{n-1}$) be a labeled graph obtained from $VQ_{n-1}$ by
inserting a zero (resp. $1$ ) in front of each vertex-labeling in
$VQ_{n-1}$. For $n>1$, $VQ_n$ is obtained by joining vertices in
$VQ^0_{n-1}$ and $VQ^1_{n-1}$, according to the rule: a vertex
$x=0x_{n-1}x_{n-2}x_{n-3}\cdots x_2x_1$ in $VQ^0_{n-1}$ and a vertex
$y=1y_{n-1}y_{n-2}y_{n-3}\cdots y_2y_1$ in $VQ^1_{n-1}$ are adjacent
in $VQ_n$ if and only if

1) $x_{n-1}x_{n-2}x_{n-3}\cdots x_2x_1=y_{n-1}y_{n-2}y_{n-3}\cdots
y_2y_1$ if $n\ne 3k$, or

2) $x_{n-3}\cdots x_2x_1=y_{n-3}\cdots y_2y_1$ and $(x_{n-1}x_{n-2},
y_{n-1}y_{n-2})\in I$ if $n=3k$, where
$I=\{(00,00),(01,01),(10,11),(11,10)\}$.

Figure~\ref{f1} shows the examples of varietal hypercubes $VQ_n$ for
$n=1, 2, 3$ and $4$.

The edges of Type $2)$ are referred to as {\it crossing edges} when
$(x_{n-1}x_{n-2}, y_{n-1}y_{n-2})\in \{(10,11),(11,10)\}$. All the
other edges are referred to as {\it normal edges}.

The varietal hypercube $VQ_n$ is proposed by Cheng and
Chuang~\cite{cc94} as an attractive alternative to the
$n$-dimensional hypercube $Q_n$ when they are used to model the
topological structure of a large-scale parallel processing system.
Like $Q_n$, $VQ_n$ is an $n$-regular graph with $2^n$ vertices and
$n2^{n-1}$ edges.

For convenience, we express $VQ_n$ as $VQ_n=L\odot R$, where
$L=VQ^0_{n-1}$ and $R=VQ^1_{n-1}$, and denote by $x_Lx_R$ the
$n$-transversal edge joining $x_L\in L$ and $x_R\in R$. The
recursive structure of $VQ_n$ gives the following simple properties.

\begin{lem}\label{lem1}\  Let $VQ_n=L\odot R$ with $n\ge
1$. Then $VQ_n$ contains no triangles and every vertex $x_L\in L$
has exactly one neighbor $x_R$ in $R$ joined by the $n$-transversal
edge $x_Lx_R$.
\end{lem}

\begin{lem}\label{lem1a}\
Let $VQ_n=L\odot R$ and $xy$ be an $n$-transversal edge in $VQ_n$
with $x\in L$ and $y\in R$. For $n\ge 3$, let $x=0abx_{n-3}\cdots
x_1$ and $\beta=x_{n-3}\cdots x_1$. Then $y=1a'b'\beta$, where
$ab=a'b'$ if $xy$ is a normal edge, and $(ab,a'b')=(1b,1\bar{b})$ if
$xy$ is a crossing edge, where $\bar{b}=\{0,1\}\setminus{b}$.
\end{lem}

\begin{lem}\label{lem2}\ Any edge in $VQ_n$ ($n\ge 2$)
is contained in a cycle of length $4$.
\end{lem}

\begin{proof}\
Clearly, the conclusion is true for $n=2$. Assume $n\ge 3$ and let
$xy$ be any edge in $VQ_n$. Then by definition of $VQ_n$ there is
some $m$ with $2\le m\le n$ such that $xy$ is an $m$-transversal
edge. Let $VQ_m=L\odot R$, $x\in L$ and $y\in R$.

If $xy$ is a normal edge, let $u_L$ be a neighbor of $x$ in $L$ and
$u_R$ be the neighbor of $u_L$ in $R$, then $y$ and $u_R$ are
adjacent and so $(x,u_L,u_R,y)$ is a cycle of length $4$.

If $xy$ is a crossing edge, let $x=01b\beta$, then
$y=11\bar{b}\beta$. Choose $u_L=01\bar{b}\beta$. Then $u_R=11b\beta$
by Lemma~\ref{lem1a}, and so $(x,u_L,u_R,y)$ is a cycle of length
$4$.
\end{proof}

\begin{lem}\label{lem3}\ Any $n$-transversal edge must be contained in some
cycle of length $5$ unless $n\ne 3k$ for $k\ge 1$.
\end{lem}

\begin{proof}\ Let $VQ_n=L\odot R$ and $xy$ be an $n$-transversal edge
in $VQ_n$, where $x\in L$ and $y\in R$. We first prove that $xy$ is
not contained in any cycle of length $5$ if $n\ne 3k$ for $k\ge 1$.
The conclusion is true for $n=1$ or $2$ clearly. Assume $n\ge 3$
below.

Suppose that there is a cycle $C=(x,u,z,v,y)$ of length 5 containing
the edge $xy$. Then $C$ contains two $n$-transversal edges. Since
$n\ne 3k$, $xy$ is a normal edge. Let $x=0ab\beta$, where
$\beta=x_{n-3}\cdots x_1$. Then $y=1ab\beta$. Since every vertex in
$L$ has exactly one neighbor in $R$ by Lemma~\ref{lem1}, $u\in L$
and $v\in R$. Without loss of generality, assume $z\in L$. Then $x$
and $z$ differ in exactly two positions. Without loss of generality,
let $z=0\bar{a}\bar{b}\beta$. Since $zv$ is an $n$-transversal edge
and $n\ne 3k$, $v=1\bar{a}\bar{b}\beta$. Thus, $y$ and $v$ differ in
exactly two positions, which implies that $y$ and $v$ are not
adjacent, a contradiction.

We now show that the $n$-transversal edge $xy$ must be contained in
some cycle of length $5$ if $n=3k$ for $k\ge 1$ by constructing such
a cycle. Let $x=0ab\beta\in L$ and $y=1a'b'\beta\in R$, where
$(ab,a'b')\in I$. A required cycle $C=(x,u,z,v,y)$ can be
constructed as follows.

If $xy$ is a normal edge, then $ab=a'b'=0b$. Let $u=00\bar{b}\beta$,
$z=01\bar{b}\beta$ and $v=11b\beta$ (where $zv$ is a crossing edge).

If $xy$ is a crossing edge, then $(ab,a'b')=(1b,1\bar{b})$. Let
$u=01\bar{b}\beta$, $z=00\bar{b}\beta$ and $v=10\bar{b}\beta$ (where
$zv$ is a normal edge).

The lemma follows.
\end{proof}

\begin{lem}\label{lem4}\ Any $n$-transversal edge in $VQ_n$ is contained in
cycles of length $6$ and $7$ for $n\ge 3$.
\end{lem}

\begin{proof}\
Let $VQ_n=L\odot R$ and $xy$ be an $n$-transversal edge in $VQ_n$,
where $x\in L$ and $y\in R$.

We first show that $xy$ is contained in a cycle of length $6$. By
Lemma~\ref{lem2}, there is a cycle $C$ of length $4$. Let
$C=(x,u,v,y)$, where $u\in L$ and $v\in R$. Also by
Lemma~\ref{lem2}, there is a cycle $C'$ of length $4$ containing the
$xu$ in $L$. Clearly, $C\cap C'=\{xu\}$. Thus, $C\cup C'-xu$ is a
cycle of length $6$ containing the edge $xy$.

We now show that $xy$ is contained in a cycle of length $7$. If
$n=3k$ for $k\ge 1$ then, by Lemma~\ref{lem3}, there is a cycle $C$
of length $5$ containing the edge $xy$. Let $C=(x,u,z,v,y)$, where
$x,u,z\in L$ and $v\in R$, without loss of generality. By
Lemma~\ref{lem2}, there is a cycle $C'$ of length $4$ containing the
edge $yv$ in $R$. Clearly, $C\cap C'=\{yv\}$. Thus $C\cup C'-yv$ is
a cycle of length $7$ containing the edge $xy$.

Assume $n\ne 3k$ for $k\ge 1$ below. In this case, all
$n$-transversal edges are normal edges. We can choose a cycle
$C=(x,u,v,y)$ such that the edge $xu$ lies on some subgraph $H$ that
is isomorphic to $VQ_{3}$. By Lemma~\ref{lem3}, there is a cycle
$C'$ of length $5$ containing the edge $xu$ in $H\subseteq L$. Then
$C\cup C'-xu$ is a cycle of length $7$ containing the edge $xy$.

The lemma follows.
\end{proof}

\vskip6pt

The {\it $n$-dimensional crossed cube} $CQ_n$ is such a graph, its
vertex-set is the same as $VQ_n$, two vertices $x=x_n\cdots x_2x_1$
and $y=y_n\cdots y_2y_1$ are linked by an edge if and only if there
exists some $j$ $(1\le j\le n)$ such that

(a) $x_n\cdots x_{j+1}=y_n\cdots y_{j+1}$,

(b) $x_j\ne y_j$,

(c) $x_{j-1}=y_{j-1}$ if $j$ is even, and

(d) $(x_{2i}x_{2i-1}, y_{2i}y_{2i-1})\in I$ for each $i=1,2,\cdots,
\left\lceil \frac 12j\right\rceil -1$.

\vskip6pt

By definitions, $VQ_n\cong CQ_n$ for each $n=1,2,3$. The following
results on $CQ_n$ are used in the proofs of our main results for
$n=3$.

\begin{lem}\label{lem6}\
{\rm (Fan {\it et al.}~\cite{fjl06}, Xu and Ma~\cite{xml06}, Yang
and Megson~\cite{ym05})}\ For any two vertices $x$ and $y$ with
distance $d$ in $CQ_n$ with $n\geq 2$, $CQ_n$ contains $xy$-paths of
every length from $d$ to $2^n-1$ except 2 when $d=1$.
\end{lem}

\begin{lem}\label{lem2.7} \ For $n\ge 3$ and any integer $\ell$
with $2^n-2\le\ell\le 2^n-1$, there exists an $xy$-path of length
$\ell$ between any pair of vertices $x$ and $y$ in $VQ_n$.
\end{lem}

\begin{proof}\
We proceed by induction on $n\ge 3$. By Lemma~\ref{lem6}, the
conclusion is true for $n=3$ since $VQ_3\cong CQ_3$. Assume the
induction hypothesis for $n-1$ with $n\ge 4$. Let $VQ_n=L\odot R$,
$x$ and $y$ be two distinct vertices in $VQ_n$.

If $x, y\in L$ (or $R$) then, by the induction hypothesis, there
exists an $xy$-path $P_L$ of length $\ell_0$ in $L$, where
$\ell_0\in\{2^{n-1}-2,2^{n-1}-1\}$. Let $u$ be the neighbor of $y$
in $P_L$, $u_R$ and $y_R$ be the neighbors of $u$ and $y$ in $R$,
respectively. By the induction hypothesis, there exists a
$u_Ry_R$-path $P_R$ of length $2^{n-1}-1$ in $R$. Then
$P_L-uy+uu_R+P_R+y_Ry$ is an $xy$-path of length $\ell_0+2^{n-1}$ in
$VQ_n$.

If $x\in L$ and $y\in R$, let $u$ be a vertex in $L$ rather than $x$
such that its neighbor $u_R$ in $R$ is different from $y$, then, by
the induction hypothesis, there exist an $xu$-path $P_L$ of length
$\ell'_0$ in $L$ and a $u_Ry$-path $P_R$ of length $2^{n-1}-1$ in
$R$, where $\ell'_0\in \{2^{n-1}-2,2^{n-1}-1\}$. Then $P_L+uu_R+P_R$
is an $xy$-path of length $\ell'_0+2^{n-1}$ in $VQ_n$.

The lemma follows.
\end{proof}

\begin{lem}\label{lem2.8}
Let $VQ_n=L\odot R$, $x_L$ and $y_L$ be two vertices in $L$. Then
$d_L(x_L,y_L)=d_R(x_R,y_R)$ if $n\ne 3k$ and
$|d_L(x_L,y_L)-d_R(x_R,y_R)|\le 2$ if $n=3k$ for $k\ge 1$.
\end{lem}

\begin{proof}
Without loss of generality, assume $d_L(x_L,y_L)\le d_R(x_R,y_R)$.
Let $P_L$ be a shortest $x_Ly_L$-path in $L$ and $P_R$ a path in $R$
obtained from $P_L$ by replacing the first position $0$ by $1$ in
every vertices. Clearly, $\varepsilon(P_R)=\varepsilon(P_L)$.

Note that for an edge $u_Lv_L$ in $P_L$, if $u_Lv_R$ is a crossing
edge, then $v_Lu_R$ is also a crossing edge. For convenience, we
call the edge $u_Lv_L$ an {\it induced crossing edge}, $u_L$ and
$v_L$ {\it induced crossing vertices}.

If both $x$ and $y$ are not induced crossing vertices, then $P_R$ is
an $x_Ry_R$-path in $R$, and so $d_R(x_R,y_R)\le
\varepsilon(P_R)=d_L(x_L,y_L)$, and so $d_R(x_R,y_R)=d_L(x_L,y_L)$.
Assume below that $\{x,y\}$ contains induced crossing vertices. Then
$n=3k$.

Let $x$ be an induced crossing vertex, $xu_L$ an induced crossing
edge. Then, $x_R$ is not an end-vertex of $P_R$, while $u_R$ is an
end-vertex of $P_R$. Similarly, if $y$ is an induced crossing
vertex, $yv_L$ an induced crossing edge, then $y_R$ is not an
end-vertex of $P_R$, while $v_R$ is an end-vertex of $P_R$. Thus, an
$x_Ry_R$-path $P'_R\subseteq P_R$ has length
 $$
 \varepsilon(P'_R)=\varepsilon(P_L)- \left\{\begin{array}{ll}
 0\ &\ {\rm if\ neither}\ x\ {\rm and}\ y\ {\rm are\ induced\ crossing\ vertices};\\
 1\ &\ {\rm if\ either}\ x\ {\rm or}\ y\ {\rm is\ an\ induced\ crossing\ vertex};\\
 2\ &\ {\rm if\ both}\ x\ {\rm and}\ y\ {\rm are\ induced\ crossing\ vertices},
 \end{array}\right.
 $$
and $d_R(x_R,y_R)\le \varepsilon(P'_R)$. If $d_R(x_R,y_R)\le
d_L(x_L,y_L)-3$ then, using the above method, we can prove that
there is an $x_Ly_L$-path $P'_L$ with length $\varepsilon(P'_L)\le
d_R(x_R,y_R)+2$, from which we have $d_L(x_L,y_L)\le
\varepsilon(P'_L)\le d_R(x_R,y_R)+2\le d_L(x_L,y_L)-1$, a
contradiction. Thus, $d_R(x_R,y_R)\ge d_L(x_L,y_L)-2$. And so
$|d_L(x_L,y_L)-d_R(x_R,y_R)|\le 2$. The lemma follows.
\end{proof}

\begin{cor}\label{cor2.9}
Let $VQ_n=L\odot R$, $x$ and $y$ be two vertices in $H$, where
$H\in\{L,R\}$. Then $d_H(x,y)=d_{VQ_n}(x,y)$.
\end{cor}

\begin{proof}
Let $x$ and $y$ be in $L$ and $P$ a shortest $xy$-path in $VQ_n$. If
$P\cap R\ne \emptyset$, then $P\cap L$ consists of several sections
of $P$. Without loss of generality, assume that $P\cap L$ consists
of two sections, $P_{xu_L}$ and $P_{v_Ly}$. Then $u_Rv_R$-section
$P_{u_Rv_R}$ of $P$ from $u_R$ to $v_R$ is in $R$. By
Lemma~\ref{lem2.8}, $d_L(u_L,v_L)\le
d_R(u_R,v_R)+2=\varepsilon(P_{u_Rv_R})+2$. Since $P$ is a shortest
$xy$-path in $VQ_n$, we have that
 $$
 \begin{array}{rl}
 d_{VQ_n}(x,y)&\le d_L(x,y)=
\varepsilon(P_{xu_L})+d_L(u_L,v_L)+\varepsilon(P_{v_Ly})\\
&\le \varepsilon(P_{xu_L})+\varepsilon(P_{u_Rv_R})+2+\varepsilon(P_{v_Ly})\\
&=\varepsilon(P)=d_{VQ_n}(x,y),
\end{array}
$$
which implies $d_L(x,y)=d_{VQ_n}(x,y)$. The corollary follows.
\end{proof}

\begin{cor}\label{cor2.10}
Let $VQ_n=L\odot R$, $x\in L$ and $y\in R$. Then there is an
$n$-transversal edge $u_Lu_R$ such that
$d_{VQ_n}(x,y)=d_L(x,u_L)+1+d_R(u_R,y)$.
\end{cor}

\section{Main Results}

A graph $G$ of order $n$ is said to be {\it $\ell$-pancyclic} (resp.
{\it $\ell$-vertex-pancyclic, $\ell$-edge-pancyclic}) if it contains
(resp. each of its vertices, edges is contained in ) cycles of every
length from $\ell$ to $n$. Clearly, an $\ell$-edge-pancyclic graph
must is $\ell$-vertex-pancyclic and $\ell$-pancyclic.

We consider edge-pancyclity of $VQ_n$. Since $VQ_n$ contains no
triangles, any edge is not contained in a cycle of length $3$.
Lemma~\ref{lem2} shows that any edge in $VQ_n$ ($n\ge 2$) is
contained in a cycle of length $4$. Lemma~\ref{lem3} shows that any
$n$-transversal edge is not contained in a cycle of length $5$ if
$n\ne 3k$ for $k\ge 1$. In general, we have the following result.

\begin{thm}\label{thm3.1}
For $n\ge 2$, every edge of $VQ_n$ is contained in cycles of every
length from $4$ to $2^n$ except $5$ and, hence, $VQ_n$ is
$6$-edge-pancyclic for $n\ge 3$.
\end{thm}

\begin{proof}\
By Lemma~\ref{lem2}, we only need to show that every edge of $VQ_n$
is contained in cycles of every length from $6$ to $2^n$ for $n\ge
3$. Let $\ell$ be an integer with $6\le \ell\le 2^n$ and $xy$ be an
edge in $VQ_n$. In order to prove the theorem, we only need to show
that $xy$ lies on a cycle of length $\ell$. We proceed by induction
on $n\ge 3$.

Since $VQ_3\cong CQ_3$, by Lemma~\ref{lem6}, the conclusion is true
for $n=3$. Assume the induction hypothesis for $n-1$ with $n\ge 4$.
Let $VQ_n=L\odot R$. There are two cases.

\vskip6pt

{\bf Case 1}\ $x, y\in L$ or $x,y\in R$. Without loss of generality,
let $x,y\in L$.

By the induction hypothesis, we only need to consider $\ell$ with
$2^{n-1}+1\le\ell\le 2^{n}$.

If $\ell=2^{n-1}+1$, then let $x_R$ and $y_R$ be the neighbors of
$x$ and $y$ in $R$, respectively. By Lemma~\ref{lem2.7}, there
exists an $x_Ry_R$-path $P_{x_Ry_R}$ of length  $2^{n-1}-2$ in $R$.
Then $xx_R+P_{x_Ry_R}+y_Ry+xy$ is a cycle of length $2^{n-1}+1$.

If $2^{n-1}+2\le\ell\le 2^{n}$, let $\ell_0=\ell-2^{n-1}-1$, then
$1\le\ell_0\le 2^{n-1}-1$. By Lemma~\ref{lem2.7}, there exists a
cycle $C$ of length $2^{n-1}$ containing the edge $xy$ in $L$. We
choose an $xz$-path $P_{xz}$ of length $\ell_0$ in $C$ that contains
$xy$. Let $x_R$ and $z_R$ be the neighbors of $x$ and $z$ in $R$,
respectively. By Lemma~\ref{lem2.7}, there exists an $x_Rz_R$-path
$P_{x_Rz_R}$ of length  $2^{n-1}-1$ in $R$. Then
$xx_R+P_{x_Rz_R}+z_Rz+P_{xz}$ is a cycle of length $\ell$ containing
the edge $xy$ in $VQ_n$.

\vskip6pt

{\bf Case 2}\ $x\in L$ and $y\in R$.

In this case, $xy$ is an $n$-transversal edge. By Lemma~\ref{lem4},
the conclusion is true for each $\ell=6,7$. Assume $\ell\ge 8$
below.

If $\ell\le 2^{n-1}+2$, let $\ell_0=\ell-2$, then $6\le\ell_0\le
2^{n-1}$. By Lemma~\ref{lem2}, there exists a 4-cycle
$C=(x,u_L,u_R,y)$. By the induction hypothesis, there exists a cycle
$C_L$ of length $\ell_0$ that contains $xu_L$ in $L$. Then, $C\cap
C_L=\{xu_L\}$, and so $C\cap C_L-\{xu_L\}$ is a cycle of length
$\ell$ containing $xy$.

If $2^{n-1}+3\le\ell\le 2^n$, let $\ell_0=\ell-2^{n-1}-1$, then
$2\le\ell_0\le 2^{n-1}-1$. Choose a vertex $u$ in $L$ rather than
$x$, By Lemma~\ref{lem2.7}, there exists an $xu$-path $P_{xu}$ of
length $2^{n-1}-1$ in $L$, from which we can choose an $xz$-path
$P_{xz}$ of length $\ell_0$. Let $z_R$ be the neighbor of $z$ in
$R$. By Lemma~\ref{lem2.7}, there exists a $z_Ry$-path $P_{z_Ry}$ of
length $2^{n-1}-1$. Thus, $P_{xz}+zz_R+P_{z_Ry}+xy$ is a cycle of
length $\ell$ containing $xy$ in $VQ_n$.

The theorem follows.
\end{proof}

\vskip6pt

A graph $G$ of order $n$ is said to be {\it panconnected} if for any
two distinct vertices $x$ and $y$ with distance $d$ in $G$ there are
$xy$-paths of every length from $d$ to $n-1$.

We consider panconnectivity of $VQ_n$. Since $VQ_n$ contains no
triangles, there exist no $xy$-paths of length two if $x$ and $y$
are adjacent. Lemma~\ref{lem3} shows that there exist no $xy$-paths
of length $4$ if $xy$ is an $n$-transversal edge in $VQ_n$ if $n\ne
3k$ for $k\ge 1$. In general, we have the following result.

\begin{thm}
For $n\ge 3$, any two vertices $x$ and $y$ in $VQ_n$ with distance
$d$, there exist $xy$-paths of every length from $d$ to $2^n-1$
except $2,4$ if $d=1$.
\end{thm}

\begin{proof}
Let $x$ and $y$ be any two vertices in $VQ_n$ with distance $d$.
First, we note that if $d=1$ then the theorem is true by
Theorem~\ref{thm3.1}. In the following discussion, we always assume
$d\ge 2$. We only need to prove that there exist $xy$-paths of every
length from $d+1$ to $2^n-1$ .

We proceed by induction on $n\ge 3$. Since $VQ_3\cong CQ_3$, by
Lemma~\ref{lem6}, the conclusion is true for $n=3$. Assume the
induction hypothesis for $n-1$ with $n\ge 4$. Let $VQ_n=L\odot R$.

\vskip6pt {\bf Case 1}. $x,y\in L$ or $x,y\in R$. Without loss of
generality, let $x,y\in L$.

By Corollary~\ref{cor2.9}, $d_L(x,y)=d$. By the induction
hypothesis, we only need to consider $\ell$ with $2^{n-1}\le \ell\le
2^n-1$.

If $2^{n-1}\le \ell\le 2^{n-1} +1$, then $2^{n-1}-2\le \ell-2\le
2^{n-1}-1$. Let $x_R$ and $y_R$ be the neighbors of $x$ and $y$ in
$R$, respectively. By Lemma~\ref{lem2.7}, there exists an
$x_Ry_R$-path $P_R$ of length $\ell-2$ in $R$. Then
$x_Ry_R+P_R+yy_R$ is an $xy$-path of length $\ell$ in $VQ_n$.

If $2^{n-1}+2\le \ell\le 2^n-1$, let $\ell_0=\ell-2^{n-1}-1$. then
$1\le \ell_0\le 2^{n-1}-2$. By Lemma~\ref{lem2.7}, there exists an
$xy$-path $P_{xy}$ of length $2^{n-1}-1$ in $L$. We choose an
$xz$-path $P_{xz}$ of length $\ell_0$ in $P_{xy}$. Clearly,
$z\notin\{x, y\}$. Let $z_R$ and $y_R$ be the neighbors of $z$ and
$y$ in $R$, respectively. By Lemma~\ref{lem2.7}, there exists a
$z_Ry_R$-path $P_R$ of length $2^{n-1}-1$ in $R$. Then
$P_{xz}+zz_R+P_R+yy_R$ is an $xy$-path of length $\ell$ in $VQ_n$.

\vskip6pt {\bf Case 2}. $x\in L$ and $y\in R$.

By Corollary~\ref{cor2.10}, there is a shortest $xy$-path $P_{xy}$
in $VQ_n$ such that $P_{xy}=P_{xu_L}+u_Lu_R+P{u_Ry}$, where $u_L\in
L$ and $u_R\in R$, $\varepsilon(P_{xu_L})=d_L(x,u_L)$ and
$\varepsilon(P_{u_Ry})=d_R(u_R,y)$. Thus,
$d=\varepsilon(P_{xu_L})+1+\varepsilon(P_{u_Ry})=d_L(x,u_L)+1+
d_R(u_R,y)$. Since $d\ge 2$, without loss of generality, assume
$d_L(x,u_L)\ge d_R(u_R,y)$.

If $d+1\le \ell\le 2^{n-1}$, let $\ell_0=\ell-d_R(u_R,y)-1$, then
$d_L(x,u_L)+1\le \ell_0\le 2^{n-1}-1$. By the induction hypothesis,
there exists an $xu_L$-path $P'$ of length $\ell_0$ in $L$. Then
$P'+u_Lu_R+P_{u_Ry}$ is an $xy$-path of length $\ell$ in $VQ_n$.

If $2^{n-1}+1\le \ell\le 2^n-1$, let $\ell_0=\ell-2^{n-1}$, then
$1\le \ell_0\le 2^{n-1}-1$. Let $y_L$ be the neighbor of $y$ in $L$.
Then $y_L\ne x$ since $x$ and $y$ are not adjacent. By
Lemma~\ref{lem2.7}, there exists an $xy_L$-path $P_{xy_L}$ of length
$2^{n-1}-1$ in $L$. We choose an $xz$-path $P_{xz}$ of length
$\ell_0$ in $P_{xy_L}$. Let $z_R$ be the neighbor of $z$ in $R$. By
Lemma~\ref{lem2.7}, there exists a $z_Ry$-path $P_{z_Ry}$ of length
$2^{n-1}-1$ in $R$. Then $P_{xz}+zz_R+P_{z_Ry}$ is an $xy$-path of
length $\ell$ in $VQ_n$.

The theorem follows.
\end{proof}

\end{document}